\documentclass[12pt,twoside]{amsart}
\usepackage{amsmath,amsthm,amsmath,amsrefs}

\usepackage{calc}
\usepackage{setspace}
\usepackage{amsbsy,amsmath,amsfonts,amsthm,amssymb,color}
\usepackage[margin=1in]{geometry}
\newtheorem{innercustomtheorem}{Theorem}
\newenvironment{customtheorem}[1]
  {\renewcommand\theinnercustomtheorem{#1}\innercustomtheorem}
  {\endinnercustomtheorem}

\theoremstyle{plain}
\newtheorem{theorem}{Theorem}%[section] %[section]
\newtheorem{lemma}[theorem]{Lemma}

\newtheorem{corollary}[theorem]{Corollary}
\newtheorem{remark}[theorem]{Remark}
\newtheorem*{SWP*}{Smith-Ward Problem}
\newtheorem{proposition}[theorem]{Proposition}

\theoremstyle{definition}

\usepackage{mathrsfs}

\newcommand{\bN}{\mathbb{N}}
\newcommand{\cT}{\mathcal{T}}

\newcommand{\cS}{\mathcal{S}}
\newcommand{\cB}{\mathcal{B}}

\newcommand{\cK}{\mathcal{K}}

\newcommand{\la}{\langle}
\newcommand{\ra}{\rangle}

\newcommand{\cH}{\mathcal{H}}
\newcommand{\cA}{\mathcal{A}}

\newcommand{\bofh}{\cB(\cH)}

\newcommand{\bC}{\mathbb{C}}

\newcommand{\id}{\text{id}}

\newcommand{\bZ}{\mathbb{Z}}
\newcommand{\cI}{\mathcal{I}}

\newcommand{\bF}{\mathbb{F}}

\newcommand{\qofh}{\mathcal{Q}(\mathcal{H})}

\usepackage{tikz}
\usepackage{tikz-cd}

\title[Four-dimensional operator systems without the lifting property]{Four-dimensional operator systems without the lifting property}

\author{Samuel J. Harris}

\address{Northern Arizona University\\
Department of Mathematics \& Statistics\\
801 S. Osborne Dr.\\
Flagstaff, AZ\\
86011 USA}
\email{samuel.harris@nau.edu}

\begin{document}

\begin{abstract}
The purpose of this note is to provide a family of explicit examples of $4$-dimensional operator systems contained in the Calkin algebra $\qofh$ on a separable infinite-dimensional Hilbert space $\mathcal{H}$ for which the identity map has no unital completely positive (ucp) lift to $\mathcal{B}(\mathcal{H})$ with respect to the canonical quotient map $\pi:\mathcal{B}(\mathcal{H}) \to \mathcal{Q}(\mathcal{H})$. More specifically, to each unital $C^*$-algebra $\mathcal{A}$ generated by $n$ unitaries and unital $*$-homomorphism $\rho:\mathcal{A} \to \mathcal{Q}(\mathcal{H})$ with no ucp lift, we construct a four-dimensional operator subsystem $\mathcal{S}$ of $M_{n+1}(\mathcal{A})$ without the lifting property. As a result, for each $n \geq 2$ we exhibit a four-dimensional operator system $\mathcal{S}$ in $M_{n+1}(C_r^*(\mathbb{F}_n))$ without the lifting property. We also obtain explicit examples where the generalized Smith-Ward problem for liftings of joint matrix ranges for three self-adjoint operators has a negative answer.
\end{abstract}

\maketitle

\section{Introduction}

For low-dimensional operator systems, a problem of great interest is determining which operator systems have the lifting property. An operator system $\cS$ is said to have the \textbf{lifting property} (or \textbf{LP} for short) if, whenever $\varphi:\cS \to \cA/\cI$ is a ucp map into a quotient of a unital $C^*$-algebra by an ideal $\cI$, then there exists a ucp map $\psi:\cS \to \cA$ wuch that $q \circ \psi=\varphi$, where $q:\cA \to \cA/\cI$ is the canonical quotient map. One can also talk about an operator system having the \textbf{local lifting property} (or \textbf{LLP} for short), which is when ucp maps $\varphi:\cS \to \cA/\cI$ always have ucp lifts to $\cA$, when restricted to finite-dimensional subsystems of $\cS$. It is well-known that every two-dimensional operator system has the lifting property. In contrast, in the three-dimensional case much less is known. Asking whether every three-dimensional operator system has the lifting property is equivalent to the \textbf{Smith-Ward problem} \cite{Ka14}. Kavruk \cite{Ka14} showed that every three-dimensional operator system has the lifting property if and only if every three-dimensional operator system is exact, and in turn, either of these holds if every three-dimensional operator system is $C^*$-nuclear. Particular attention has been paid to operator subsystems of the Calkin algebra $\qofh=\bofh/\cK(\cH)$, where $\cH$ is a separable and infinite-dimensional Hilbert space. Let $\pi:\bofh \to \qofh$ denote the canonical quotient map. Determining whether a three-dimensional operator system $\cS=\text{span}\{\pi(I),\pi(T),\pi(T)^*\} \subseteq \qofh$ has the lifting property is equivalent to determining whether the identity map $\id:\cS \to \qofh$ has a ucp lift to $\bofh$ \cite[Proposition~11.4]{Ka14}.

Past work has dealt with this problem of determining whether the identity map $\id:\cS \to \qofh$ of a finite-dimensional subsystem of $\qofh$ has a ucp lift to $\bofh$. Paulsen \cite{Pa82} constructed a five-dimensional operator subsystem of $\qofh$ for which this fails. Kavruk \cite{Ka14} constructed a four-dimensional operator subsystem of $M_6$ that does not have the lifting property--in turn, one can obtain a four-dimensional operator subsystem $\cS$ of $\qofh$ where the identity map has no ucp lift (see Remark \ref{remark: 4-dimensional from M6}). However, in Kavruk's proof the ucp map $\varphi$ that has no ucp lift is not explicit, and although the map can be arranged to have codomain $\qofh$ while having no ucp lift to $\bofh$, the ucp map involved is still not explicit. One can replace $\cS$ with $\varphi(\cS)$ and get a subsystem of $\qofh$ of dimension at most four where the identity map has no ucp lift; however, as the map $\varphi$ is implicit, the operator system $\varphi(\cS)$ is also implicit. Related to these works is the recent note of Pop \cite{Pop20}, which proves the existence of finite-dimensional subsystems of $\qofh$ without the lifting property via alternative means; however, those subsystems are also implicit.

The goal of this note is to provide a family of explicit examples of $4$-dimensional subsystems of the Calkin algebra without the lifting property. The main result in this note relies on the existence of a finitely generated unital $C^*$-algebra $\cA$ that has a unital $*$-homomorphism into the Calkin algebra with no ucp lift to $\bofh$ (this was first proven by Anderson \cite{An78}).

\begin{customtheorem}{\ref{theorem: main}}
Let $\mathcal{A}$ be a finitely generated unital $C^*$-algebra for which there exists a unital $*$-homomorphism $\rho:\cA \to \qofh$ that has no ucp lift to $\bofh$. Then there is $n \in \bN$ and a four-dimensional operator subsystem $\cS$ of $M_{n+1}(\cA)$ for which $(\rho \otimes \id_{n+1})_{|\cS}:\cS \to \qofh \otimes M_{n+1}$ has no ucp lift to $\bofh \otimes M_{n+1}$. In particular, $\cS$ does not have the lifting property.
\end{customtheorem}

Paulsen's example \cite{Pa82} appeared soon after Anderson constructed a unital $C^*$-subalgebra $\cA$ of $\qofh$ where the identity map has no ucp lift to $\bofh$. This algebra is generated by two unitaries $u_1,u_2$ and a projection $p$ such that $C^*(\{u_1,u_2\})$ is isomorphic to the reduced group $C^*$-algebra $C_r^*(\bF_2)$ of the free group on two generators. The arguments in \cite{An78} rely only on the facts that $\bF_2$ is residually finite and non-amenable (so that its group von Neumann algebra is not injective \cite{CE77a,CE77b}), and work just as well when replacing $\bF_2$ with a group of the form $G=G_1*G_2$ where $G_1$ and $G_2$ are countable discrete groups satisfying $|G_1| \geq 2$, $|G_2| \geq 2$ and $|G_1|+|G_2| \geq 5$, since such groups contain a copy of $\bF_2$.

We take a slightly different approach than in \cite{Pa82}, by using some of the bimodule properties of the multiplicative domain of a ucp map. For a ucp map $\varphi:\cA \to \cB$, where $\cA$ and $\cB$ are unital $C^*$-algebras, Choi's inequality \cite{Ch74} shows that $\varphi(a)^*\varphi(a) \leq \varphi(a^*a)$ and $\varphi(a)\varphi(a)^* \leq \varphi(aa^*)$ . Moreover, the \textbf{left multiplicative domain} of $\varphi$ is
\[ \mathfrak{m}_{\ell}(\varphi)=\{ a \in \cA: \varphi(a)\varphi(a)^*=\varphi(aa^*)\}=\{ a \in \cA: \varphi(ac)=\varphi(a)\varphi(c), \, \forall c \in \cA\}.\]
Similarly the \textbf{right multiplicative domain} of $\varphi$ is
\[ \mathfrak{m}_r(\varphi)=\{ a \in \cA: \varphi(a)^*\varphi(a)=\varphi(a^*a)\}=\{ a \in \cA: \varphi(ca)=\varphi(c)\varphi(a), \, \forall c \in \mathcal{A}\}.\]
The \textbf{multiplicative domain} of $\varphi$ is $\mathfrak{m}(\varphi)=\mathfrak{m}_{\ell}(\varphi) \cap \mathfrak{m}_r(\varphi)$, which is a $C^*$-subalgebra of $\cA$; moreover, $\varphi_{|\mathfrak{m}(\varphi)}$ is a unital $*$-homomorphism \cite{Ch74}. In yielding Theorem \ref{theorem: main}, our arguments rely on the left multiplicative domain of a ucp map (as well as the two-sided version; see Lemma \ref{lemma: multiplicative domain}).

We also yield several applications of Theorem \ref{theorem: main}. One is that, whenever $\cT$ is a finite-dimensional operator system without the lifting property, then there is a four-dimensional operator subsystem of $M_{n+1}(C_u^*(\cT))$ without the lifting property, where $C_u^*(\cT)$ is the universal $C^*$-algebra of $\cT$ and $n \leq 4 \dim(\cT)$; see Corollary \ref{corollary: universal}. Another application of Theorem \ref{theorem: main} is an explicit four-dimensional operator subsystem of $\qofh$ where the identity map has no ucp lift to $\bofh$ (Corollary \ref{corollary: ext}). We also prove that, for each $n \geq 2$, there is a four-dimensional subsystem of $M_{n+1}(C_r^*(\bF_n))$ without the lifting property (Corollary \ref{corollary: subsystem of free group algebra}). Lastly, all of these examples bear weight on a certain generalization of the Smith-Ward problem, which we explain below.

The Smith-Ward problem stems initially from the \textbf{Smith-Ward Theorem} regarding how the first $N$ essential matrix ranges of an operator can be witnessed exactly by the first $N$ matrix ranges of a certain lift of that operator. For an element $T$ of a unital $C^*$-algebra $\cA$ and for $n \in \bN$, one can define the \textbf{$n$-th matrix range} of $T$ to be the set $W^n(T)=\{\Phi(T): \Phi \in \text{UCP}(\cA,M_n)\}$. By Arveson's extension theorem \cite{Ar69}, $W^n(T)$ does not depend on the $C^*$-algebra containing $T$. (We note that, if $n=1$ and $\cA \subseteq \bofh$, then $W^1(T)$ is the closure of the \textbf{numerical range} of $T$, which is the set of all complex numbers of the form $\la Tx,x \ra$, where $x \in \cH$ and $\|x\|=1$.) We refer the reader to works of Arveson \cite{Ar70,Ar72} for some of the history and development regarding matrix ranges, and works of M\"uller \cite{Mu10} and Narcowich and Ward \cite{NW82} for more information. Smith and Ward \cite{SW80} proved that, if $T \in \bofh$ and $N \in \bN$, then there exists a compact operator $K \in \cK(\cH)$ such that $W^n(\pi(T))=W^n(T+K)$ for all $1 \leq n \leq N$, where $\pi:\bofh \to \qofh$ again denotes the canonical quotient map. This theorem was recently generalized by Li, Paulsen and Poon to the $n$-th joint matrix range of a tuple of operators in $\bofh$ \cite{LPP19}. For a tuple $A_1,...,A_m$ of operators in a unital $C^*$-algebra $\cA$, the \textbf{$n$-th joint matrix range} of $(A_1,...,A_n)$ is
\[ W^n(A_1,...,A_m)=\{(\Phi(A_1),...,\Phi(A_m)): \Phi \in \text{UCP}(\cA,M_n)\},\]
which again does not depend on the unital $C^*$-algebra $\cA$ containing $A_1,...,A_n$. Then Li, Paulsen and Poon prove \cite{LPP19} that, whenever $m \in \bN$ and $A_1,...,A_m \in \bofh$ are self-adjoint and $N \in \bN$, there exist compact operators $K_1,...,K_m$ such that
\[ W^n(\pi(A_1),...,\pi(A_m))=W^n(A_1+K_1,...,A_m+K_m), \, \forall 1 \leq n \leq N.\]
When $m=1$, the problem of whether one can drop the dependence on $N$ above is still open:

\begin{SWP*}
If $T \in \bofh$, then is there a $K \in \cK(\cH)$ such that $W^n(T+K)=W^n(\pi(T))$ for all $n \in \bN$?
\end{SWP*}

The Smith-Ward problem, although only about ucp maps into matrix algebras, can be viewed as a lifting problem. In fact, \cite[p.~304]{Pa82} shows that the Smith-Ward problem is a problem regarding three-dimensional operator systems. In particular, the Smith-Ward problem has a positive answer for all operators in $\qofh$ if and only if every three-dimensional operator system has the lifting property \cite{Ka14}. Moreover, the lifting problem for a three-dimensional operator subsystem $\cS=\text{span}\{\pi(I),\pi(T),\pi(T)^*\} \subseteq \qofh$ is equivalent to the property that the identity map $\id:\cS \to \qofh$ has a ucp lift into $\bofh$ \cite[Proposition~11.4]{Ka14}. In turn, this assertion is equivalent to the existence of $K_1,K_2 \in \cK(\cH)$ such that $W^n(\pi(A_1),\pi(A_2))=W^n(A_1+K_1,A_2+K_2)$ for all $n \in \bN$, where $A_1=\text{Re}(T)$ and $A_2=\text{Im}(T)$. Similarly, for a tuple $(A_1,...,A_m) \in \bofh^m$, determining whether there are $K_1,...,K_m \in \cK(\cH)$ such that $W^n(\pi(A_1),...,\pi(A_m))=W^n(A_1+K_1,...,A_m+K_m)$ for all $n \in \bN$ is equivalent to determining whether 
\[\text{id}: \text{span}\{I,\pi(A_1),...,\pi(A_m),\pi(A_1)^*,...,\pi(A_m)^*\} \to \qofh\] 
has a ucp lift to $\bofh$. (This fact is well-known and follows directly from Proposition \ref{proposition: generalized SWP and ucp lift}.) This problem is sometimes referred to as the \textbf{generalized Smith-Ward problem} for the tuple $(A_1,...,A_m)$. (As noted in \cite{LPP19}, it suffices to only consider the case when $A_1,...,A_m$ are self-adjoint.) All the four-dimensional examples in this note yield examples of self-adjoint operators $T_1,T_2,T_3 \in \bofh$ such that, whenever $K_1,K_2,K_3 \in \cK(\cH)$, there exists an $m \in \bN$ for which $W^m(\pi(T_1),\pi(T_2),\pi(T_3)) \neq W^n(T_1+K_1,T_2+K_2,T_3+K_3)$. Hence, we yield a new collection of explicit counterexamples to the generalized Smith-Ward problem for three self-adjoint operators. While such examples for three self-adjoint operators are known to exist by work of Kavruk \cite{Ka14}, the existence of such examples rely on the non-exactness of the full group $C^*$-algebra $C^*(*_3\bZ_2)$ of the free product of three copies of $\bZ_2$, and so the examples are not explicit. We note in passing that Corollary \ref{corollary: self-adjoint partial isometry and two projections} gives an example where $T_1$ is a self-adjoint partial isometry (i.e. the difference of two mutually orthogonal projections) and $T_2,T_3$ are projections for which the generalized Smith-Ward problem has a negative answer. We note that the Smith-Ward problem, which is still open, is for two self-adjoint operators (namely, the real and imaginary parts of a single operator $T \in \bofh$), which involves three-dimensional operator systems.

\section{Main results}

We start by showing that, if $\cA$ is a unital $C^*$-algebra generated by $n$ unitaries, then a certain $4$-dimensional operator subsystem of $M_{n+1}(\cA)$ enjoys a unique extension property with respect to homomorphisms on $\cA$. The matrix entries of the spanning elements of $\cS$ actually lie in the operator system spanned by the unitary generators of $\cA$. In this lemma and the rest of the paper, we make the identification $M_{n+1}(\cA) \simeq \cA \otimes M_{n+1}$.

\begin{lemma}
\label{lemma: multiplicative domain}
Suppose that $\cA$ is a unital $C^*$-algebra generated by $n$ unitaries $u_1,...,u_n$. Define 
\[ S=\sum_{j=1}^{n} (u_j \otimes E_{j,j+1} + u_j^* \otimes E_{j+1,j}) \text{ and } J=\sum_{i,j=1}^{n+1} 1 \otimes E_{i,j}.\]
Let $\cS=\text{span}\{1 \otimes I_{n+1},S,1 \otimes E_{11}, J\}$, and let $\rho:\cA \to \cB$ be a unital $*$-homomorphism into a unital $C^*$-algebra $\cB$. If $\psi:\cA \otimes M_{n+1} \to \cB \otimes M_{n+1}$ is a ucp map satisfying $\psi_{|\cS}=(\rho \otimes \id_{n+1})_{|\cS}$, then $\psi=\rho \otimes \id_{n+1}$ on $\cA \otimes M_{n+1}$.
\end{lemma}

\begin{proof}
Define $a_0=1$ and, for each $1 \leq k \leq n$, define $a_k=u_1 \cdots u_k$. To simplify notation, we let $\eta=\rho \otimes \id_{n+1}$. Since $1 \otimes E_{1,1} \in \cS$ is a projection and $\eta$ is a $*$-homomorphism, $\psi(1 \otimes E_{1,1})=\eta(1 \otimes E_{1,1})$ is a projection. It follows that $a_0 \otimes E_{1,1}=1 \otimes E_{1,1}$ belongs to the multiplicative domain $\mathfrak{m}(\psi)$ of $\psi$. Using the fact that $\psi(S)=\eta(S)$, it follows that
\[ \psi(a_1 \otimes E_{1,2})=\psi((1 \otimes E_{1,1})S)=\eta(1 \otimes E_{1,1})\eta(S)=\eta((1 \otimes E_{1,1})S)=\eta(a_1 \otimes E_{1,2}).\]
Therefore, since $\eta$ is a $*$-homomorphism,
\[\psi(a_1 \otimes E_{1,2})\psi(a_1 \otimes E_{1,2})^*=\eta((a_1 \otimes E_{1,2})(a_1 \otimes E_{1,2})^*)=\eta(1 \otimes E_{1,1})=\psi((a_1 \otimes E_{1,2})(a_1 \otimes E_{1,2})^*).\]
Hence, $a_1 \otimes E_{1,2}$ belongs to the left multiplicative domain $\mathfrak{m}_{\ell}(\psi)$ of $\psi$. Working inductively, suppose that $k \in \{1,...,n-1\}$ and that $\psi(a_i \otimes E_{1,i+1})=\eta(a_i \otimes E_{1,i+1})$ for each $i \in \{0,1,...,k\}$. We will show that $\psi(a_{k+1} \otimes E_{1,k+2})=\eta(a_{k+1} \otimes E_{1,k+2})$. First, note that if $0 \leq i \leq k$, then
\[ \psi(a_i \otimes E_{1,i+1})\psi(a_i \otimes E_{1,i+1})^*=\eta(a_i \otimes E_{1,i+1})\eta(a_i \otimes E_{1,i+1})^*=\eta(1 \otimes E_{1,1})=\psi(1 \otimes E_{1,1}),\]
so that $\psi(a_i \otimes E_{1,i+1})\psi(a_i \otimes E_{1,i+1})^*=\psi((a_i \otimes E_{1,i+1})(a_i \otimes E_{1,i+1})^*)$. Thus, $a_i \otimes E_{1,i+1} \in \mathfrak{m}_{\ell}(\psi)$ for $0 \leq i \leq k$. Then one has
\begin{align*}
\psi(a_{k+1} \otimes E_{1,k+2}+a_{k-1} \otimes E_{1,k})&=\psi((a_k \otimes E_{1,k+1})S) \\
&=\eta(a_k \otimes E_{1,k+1})\eta(S) \\
&=\eta((a_k \otimes E_{1,k+1})S) \\
&=\eta(a_{k+1} \otimes E_{1,k+2}+a_{k-1} \otimes E_{1,k}).
\end{align*}
Since $\eta$ and $\psi$ agree on the element $a_{k-1} \otimes E_{1,k}$, they also agree on $a_{k+1} \otimes E_{1,k+2}$ by linearity. By induction, it follows that $\psi(a_k \otimes E_{1,k+1})=\eta(a_k \otimes E_{1,k+1})$ for all $k=1,...,n$. But $(a_k \otimes E_{1,k+1})(a_k \otimes E_{1,k+1})^*=1 \otimes E_{1,1}$ and $\psi(1 \otimes E_{1,1})=\eta(1 \otimes E_{1,1})$, so $a_k \otimes E_{1,k+1} \in \mathfrak{m}_{\ell}(\psi)$ for each $k$. 

We wish to show that $a_k \otimes E_{1,k+1}$ also belongs to the right multiplicative domain $\mathfrak{m}_r(\psi)$ of $\psi$. To this end, we observe that, since $\eta(a_k \otimes E_{1,k+1})=\psi(a_k \otimes E_{1,k+1})$ for all $1 \leq k \leq n$,
\begin{align*}
\sum_{k=1}^n \eta(1 \otimes E_{k+1,k+1})&=\sum_{k=1}^n \eta(a_k \otimes E_{1,k+1})^*\eta(a_k \otimes E_{1,k+1}) \\
&=\sum_{k=1}^n \psi(a_k \otimes E_{1,k+1})^*\psi(a_k \otimes E_{1,k+1}) \\
&\leq \sum_{k=1}^n\psi((a_k \otimes E_{1,k+1})^*(a_k \otimes E_{1,k+1})) \\
&=\sum_{k=1}^n \psi(1 \otimes E_{k+1,k+1}).
\end{align*}
Since $\eta(1 \otimes E_{1,1})=\psi(1 \otimes E_{1,1})$ and $\psi$ is unital, adding $\eta(1 \otimes E_{1,1})$ to both sides of the above inequality yields $1 \otimes I_{n+1}$. Thus, the inequality above is an equality. This forces $\psi(a_k \otimes E_{1,k+1})^*\psi(a_k \otimes E_{1,k+1})=\psi(1 \otimes E_{k+1,k+1})$ for all $1 \leq k \leq n$, so $a_k \otimes E_{1,k+1} \in \mathfrak{m}_{\ell}(\psi) \cap \mathfrak{m}_r(\psi)=\mathfrak{m}(\psi)$. Hence, $\psi$ is both a $*$-homomorphism and equal to $\eta$ on the subalgebra $C^*(\{ a_k \otimes E_{1,k+1}\}_{k=0}^n)$. In particular, since each $a_k$ is unitary, $1 \otimes E_{j,j}$ belongs to $\mathfrak{m}(\psi)$ for each $1 \leq j \leq n+1$. It follows that $\psi$ is a bimodule map over the diagonal subalgebra $1 \otimes \mathcal{D}_{n+1}$; that is, for every $i,j=1,...,n+1$ and $X \in \cA \otimes M_{n+1}$, we have
\[ \psi((1 \otimes E_{i,i})X(1\otimes E_{j,j}))=\psi(1 \otimes E_{i,i})\psi(X)\psi(1 \otimes E_{j,j}).\]

Finally, since $\psi(J)=\eta(J)$ it follows that, for each $1 \leq i,j \leq n+1$,
\[ \psi(1 \otimes E_{i,j})=\psi((1 \otimes E_{i,i})J(1 \otimes E_{j,j}))=\eta(1 \otimes E_{i,i})\eta(J)\eta(1 \otimes E_{j,j})=\eta((1 \otimes E_{i,i})J(1 \otimes E_{j,j})),\]
so that $\psi(1 \otimes E_{i,j})=\eta(1 \otimes E_{i,j})$ for all $i,j$. Since we already know that $\psi(1 \otimes E_{i,i})=\eta(1 \otimes E_{i,i})$ and $\psi(1 \otimes E_{j,j})=\eta(1 \otimes E_{j,j})$, it follows that $1 \otimes E_{i,j}$ is in the multiplicative domain of $\psi$ and the maps $\psi$ and $\eta$ agree on this element. This is enough to show that $\psi(a_k \otimes E_{i,j})=\eta(a_k \otimes E_{ij})$ for all $1 \leq k \leq n$ and $1 \leq i,j \leq n+1$, and that each of the elements $a_k \otimes E_{i,j}$ is in $\mathfrak{m}(\psi)$. Since $\mathcal{A} \otimes M_{n+1}$ is generated as a $C^*$-algebra by such elements, it follows that $\psi=\eta$.
\end{proof}

\begin{theorem}
\label{theorem: main}
Let $\mathcal{A}$ be a finitely generated unital $C^*$-algebra for which there exists a unital $*$-homomorphism $\rho:\cA \to \qofh$ that has no ucp lift to $\bofh$. Then there is $n \in \bN$ and a four-dimensional operator subsystem $\cS$ of $M_{n+1}(\cA)$ for which $(\rho \otimes \id_{n+1})_{|\cS}:\cS \to \qofh \otimes M_{n+1}$ has no ucp lift to $\bofh \otimes M_{n+1}$. In particular, $\cS$ does not have the lifting property.
\end{theorem}

\begin{proof}
Let $\pi:\bofh \to \qofh$ be the canonical quotient map, and $\cS$ be the operator system in Lemma \ref{lemma: multiplicative domain}. We will show that there is no ucp lift of $(\rho \otimes \id_{n+1})_{|\cS}$ into $\bofh \otimes M_{n+1}$. Indeed, suppose that there is a ucp map $\varphi:\cS \to \bofh \otimes M_{n+1}$ such that $(\pi \otimes id_{M_{n+1}}) \circ \varphi=(\rho \otimes \id_{n+1})_{|\cS}$. By Arveson's extension theorem \cite{Ar69}, we may extend $\varphi$ to a ucp map defined from $\cA \otimes M_{n+1}$ into $\bofh \otimes M_{n+1}$, which we still denote by $\varphi$. Then we define $\psi=(\pi \otimes \id_{M_{n+1}}) \circ \varphi:\cA \otimes M_{n+1} \to \qofh \otimes M_{n+1}$, which is ucp. Since $\psi_{|\cS}=(\rho \otimes \id_{n+1})_{|\cS}$, Lemma \ref{lemma: multiplicative domain} shows that $\psi=\rho \otimes \id_{n+1}$. But then the mapping $\chi:\cA \to \qofh$ defined by $\chi(x)=(I_{\cH} \otimes E_{11})\varphi(x \otimes I_{n+1})(I_{\cH} \otimes E_{11})$ is a unital completely positive lift of $\rho$ into $\bofh$, which is a contradiction. Thus, no such ucp lift exists for $(\rho \otimes \id_{n+1})_{|\cS}$ into $\bofh \otimes M_{n+1}$. The result follows.
\end{proof}

We first use Theorem \ref{theorem: main} to show that we can construct $4$-dimensional operator systems without LP from any finite-dimensional operator system that does not have LP.

\begin{corollary}
\label{corollary: universal}
Suppose that $\cT$ is a finite-dimensional operator system without the lifting property. Then there is an $n \in \bN$ with $n \leq 4\dim(\cT)$, and a $4$-dimensional operator system $\cS \subseteq M_{n+1}(C_u^*(\cT))$ without the LP.
\end{corollary}

\begin{proof}
Since $\cT$ does not have the lifting property, there is a ucp map $\varphi:\cT \to \qofh$ with no ucp lift to $\bofh$ \cite[Proposition~7.4]{Ka14}. By the universal property of the universal $C^*$-algebra of an operator system, there is a unique unital $*$-homomorphism $\rho:C_u^*(\cT) \to \qofh$ such that $\rho_{|\cT}=\varphi$. Since $\varphi$ has no ucp lift to $\bofh$, the map $\rho$ does not have a ucp lift to $\bofh$ either. Since $C_u^*(\cT)$ is a $C^*$-cover of $\cT$, any spanning set for $\cT$ generates $C_u^*(\cT)$ as a $C^*$-algebra. As every element of a unital $C^*$-algebra is a linear combination of at most $4$ unitaries and $\dim(\cT)<\infty$, there is a generating set of $n$ elements for $C_u^*(\cT)$ where $n \leq 4\dim(\cT)$. Then Theorem \ref{theorem: main} guarantees the existence of a four-dimensional operator subsystem $\cS$ of $M_{n+1}(\cT)$ for which $(\rho \otimes \id_{n+1})_{|\cS}$ has no ucp lift to $\bofh \otimes M_{n+1}$. Hence, $\mathcal{S}$ does not have the LP.
\end{proof}

Unfortunately the universal $C^*$-algebra of an operator system, even in small dimensions, is very large and in general not tractable. Hence, of particular interest are cases where the operator subsystem of $\cA$ has nice algebraic properties in the $M_{n+1}(\cA)$. Of particular interest is the case when the unital $*$-homomorphism $\rho:\cA \to \qofh$ can be arranged to be injective, while still having no ucp lift to $\bofh$. Such a homomorphism corresponds to an element of the $\text{Ext}$ semigroup of $\cA$ that is not invertible (see \cite{Ar77}). While it is, in general, hard to find examples of $C^*$-algebras whose Ext semigroup is not a group, there are a few nice families of examples (see the discussion before Corollary \ref{corollary: subsystem of free group algebra}). In general, the next corollary follows immediately from Theorem \ref{theorem: main}.

\begin{corollary}
\label{corollary: ext}
Suppose that $\mathcal{A}$ is a finitely generated unital $C^*$-algebra such that $\text{Ext}(\cA)$ is not a group. Let $\rho:\cA \to \qofh$ be an injective unital $*$-homomorphism with no ucp lift to $\bofh$. Then there exists an $n \in \bN$, a $4$-dimensional operator subsystem $\mathcal{S}$ of $M_{n+1}(\cA)$ and a unitary $V:\cH \otimes \bC^{n+1} \to \cH$ such that the mapping $V(\rho \otimes \id_{n+1})(\cdot)V^*:\mathcal{S} \to \qofh$ has no ucp lift to $\bofh$.
\end{corollary}

\begin{proof}
Any injective unital $*$-homomorphism $\rho:\mathcal{A} \to \qofh$ defines, up to equivalence, an element $[\rho]$ of $\text{Ext}(\cA)$. This element has a ucp lift to $\bofh$ if and only if $[\rho]$ is invertible in $\text{Ext}(\cA)$ \cite{Ar77}. Since $\text{Ext}(\cA)$ is not a group, we can choose $\rho$ so that $[\rho]$ is not invertible in $\text{Ext}(\cA)$. By Theorem \ref{theorem: main}, there is an $n \in \bN$ and a four-dimensional subsystem $\cS$ of $M_{n+1}(\cA)$ for which $(\rho \otimes \id_{n+1})_{|\cS}$ has no ucp lift to $M_{n+1}(\bofh) \simeq \bofh \otimes M_{n+1}$. Since $\cH$ is assumed to be separable and infinite-dimensional, we can choose a unitary $V:\cH \otimes \bC^{n+1} \to \cH$. Then $\bofh \otimes M_{n+1}$ is unitarily equivalent to $\bofh$, and $\cK(\cH) \otimes M_{n+1}$ is unitarily equivalent to $\cK(\cH)$ (via $V$), so the same holds for $\qofh \otimes M_{n+1}$ and $\qofh$. Since $\rho \otimes \id_{n+1}$ has no ucp lift to $\bofh \otimes M_{n+1}$, the map $V(\rho \otimes \id_{n+1})(\cdot)V^*$ has no ucp lift to $\bofh$.
\end{proof}

Once one has an operator system $\cS$ and a ucp map $\varphi:\cS \to \qofh$ without a ucp lift to $\bofh$, one can arrange to have a subsystem of the Calkin algebra where the identity map has no ucp lift to $\bofh$. Indeed, if $\psi:\varphi(\cS) \to \bofh$ was a ucp map such that $\pi \circ \psi=\id_{\varphi(\cS)}$, then $\pi \circ (\psi \circ \varphi)=\varphi$ on $\cS$, so that $\psi \circ \varphi$ is a ucp lift of $\varphi$ to $\bofh$, a contradiction. So, by replacing $\cS$ with $\varphi(\cS)$, we can always arrange for the non-lifting ucp map to be the identity map. In our setting, Corollary \ref{corollary: ext} yields examples where this map $\varphi$ is a complete order isomorphism (in particular, $\dim(\cS)=\dim(\varphi(\cS))$). While already known, the next result immediately follows.

\begin{corollary}
There exist $4$-dimensional operator subsystems of $\qofh$ without the lifting property.
\end{corollary}

\begin{remark}
\label{remark: 4-dimensional from M6}
The existence of a $4$-dimensional operator system without the lifting property is not new; an example was given by Kavruk \cite{Ka15} of such an operator system (in fact, that operator system is a nuclearity detector and hence cannot have the LP). Kavruk's operator system is contained in $M_6$, and for $\cH$ infinite-dimensional, there are embeddings of $M_n$ into $\qofh$. Indeed, one can write $\cH \simeq \bigoplus_{i=1}^n \cH_i$ as Hilbert spaces, where $\cH_i=\cH$ for all $i$. Then take $F_i$ to be the orthogonal projection onto the $i$-th copy of $\cH$, and define $F_{ij}$ to be the partial isometry from $\cH_j$ to $\cH_i$ that sends $h \in \cH_j$ to the same $h$, now regarded in $\cH_i$. Then the mapping $E_{ij} \mapsto F_{ij}$ yields an injective $*$-homomorphism $M_n \to \bofh$, that clearly misses $\cK(\cH)$. Thus, $M_n$ embeds into $\qofh$. Applying this embedding to Kavruk's example yields a $4$-dimensional operator subsystem $\mathcal{S}$ of $\qofh$ without the LP. For that operator system, there is a ucp map $\varphi:\mathcal{S} \to \qofh$ with no ucp lift to $\bofh$. However, the map is not explicit. One can make the map explicit by replacing $\cS$ with $\varphi(\cS)$, at the cost of losing information about the original operator system, since $\varphi$ need not keep track of any of the algebraic properties of a basis for $\cS$.
\end{remark}

Going through the approach in this note yields that our non-lifting ucp map on the four-dimensional subsystem of $\cS$ preserves algebraic properties of the generators of $\cA$.

The nicest family of $C^*$-algebras in this setting are the reduced group $C^*$-algebras of free groups on finitely many generators; i.e., $C_r^*(\mathbb{F}_n)$ where $n \in \bN$ and $n \geq 2$. A groundbreaking result of Haagerup and Th\o{}rbjornsen \cite{HT05} shows that $\text{Ext}(C_r^*(\mathbb{F}_n))$ is not a group. As a result, the fact that $C_r^*(\mathbb{F}_n)$ fails the local lifting property can be witnessed at the level of $4$-dimensional operator systems (of matrix algebras over $C_r^*(\mathbb{F}_n)$). The result on $\text{Ext}(C_r^*(\bF_n))$ was extended to all non-abelian limit groups in \cite{LMH25}. A group $G$ is a \textbf{limit group} if it is finitely generated and, for each finite subset $S$ of $G$, there exists a group homomorphism $\Phi: G \to \mathbb{F}$ into a free group $\mathbb{F}$ such that $\Phi_{|S}$ is injective. For such groups, the failure of the LLP of the reduced group $C^*$-algebra can be witnessed at the four-dimensional operator system level.

\begin{corollary}
\label{corollary: subsystem of free group algebra}
Let $n \in \bN$ with $n \geq 2$ and let $u_1,...,u_n$ be unitary generators of $C_r^*(G)$ where $G$ is a non-abelian limit group (eg. $G=\mathbb{F}_n$ for $n \geq 2$). Then the operator system
\[ \mathcal{S}=\text{span}\{ 1 \otimes I_{n+1}, \sum_{j=1}^n (u_j \otimes E_{j,j+1}+u_j^* \otimes E_{j+1,j}),\sum_{i,j=1}^{n+1} 1 \otimes E_{i,j},1 \otimes E_{1,1}\}\]
contained in $M_{n+1}(C_r^*(G))$ does not have the lifting property. Moreover, there exists an injective unital $*$-homomorphism $\rho:C_r^*(G) \to \qofh$ for which $(\rho \otimes \id_{n+1})_{|\cS}$ does not have a ucp lift to $\bofh$.
\end{corollary}

Finding finite-dimensional operator subsystems of $C_r^*(\mathbb{F}_n)$, or matrices over $C_r^*(\mathbb{F}_n)$, without the lifting property, is not new, but reducing it to the $4$-dimensional case appears to be new. In $C_r^*(\mathbb{F}_n)$, one has that the $(2n+1)$-dimensional operator subsystem \[\mathcal{S}_r(\mathbb{F}_n):=\text{span}\{1,u_1,...,u_n,u_1^*,...,u_n^*\}\] 
does not have the lifting property. Indeed, the local lifting property for an operator system $\cT$ is equivalent to there being a unique operator system tensor product structure on $\cT \otimes \bofh$ \cite[Theorem~8.6]{KPTT13}. So, if $\mathcal{S}_r(\mathbb{F}_n)$ had the lifting property, then we would have $\mathcal{S}_r(\mathbb{F}_n) \otimes_{\min} \bofh=\mathcal{S}_r(\mathbb{F}_n) \otimes_{\max} \bofh$ (i.e. the formal identity map would be a complete order isomorphism). But since $\mathcal{S}_r(\mathbb{F}_n)$ ``contains enough unitaries" in $C_r^*(\mathbb{F}_n)$ (i.e. the unitary elements in the operator system generate $C_r^*(\mathbb{F}_n)$), it would follow that $C_r^*(\mathbb{F}_n) \otimes_{\min} \bofh=C_r^*(\mathbb{F}_n) \otimes_{\max}\bofh$ \cite{KPTT13,Ka14}, which is equivalent to $C_r^*(\mathbb{F}_n)$ having the local lifting property \cite{Ki93}, a contradiction. Thus, $\mathcal{S}_r(\mathbb{F}_n)$ does not have the lifting property. However, even when $n=2$ one has $\dim(\mathcal{S}_r(\mathbb{F}_2))=5$, so Corollary \ref{corollary: subsystem of free group algebra} is still a reduction in that case.

The situation when $n=2$ has other benefits. In this case, our operator subsystem of $M_3(C_r^*(\mathbb{F}_2))$ that does not have the LP is given by
\begin{equation}
\mathcal{R}=\text{span} \left\{ \begin{pmatrix} 1 & 0 & 0 \\ 0 & 1 & 0 \\ 0 & 0 & 1 \end{pmatrix},\begin{pmatrix} 0 & u_1 & 0 \\ u_1^* & 0 & u_2 \\ 0 & u_2^* & 0 \end{pmatrix},\begin{pmatrix} 1 & 0 & 0 \\ 0 & 0 & 0 \\ 0 & 0 & 0 \end{pmatrix},\begin{pmatrix} 1 & 1 & 1 \\ 1 & 1 & 1 \\ 1 & 1 & 1 \end{pmatrix}\right\}. \label{subsystem of M3 of C_r*(F_2)}
\end{equation}
We note that $V:=\frac{1}{\sqrt{2}}\begin{pmatrix} 0 & u_1 & 0 \\ u_1^* & 0 & u_2 \\ 0 & u_2^* & 0 \end{pmatrix}$ is a self-adjoint partial isometry; that is, $V^*=V$ and $V^3=V$. In particular, $V=P-Q$ for two mutually orthogonal projections $P,Q$ in $M_3(C_r^*(\mathbb{F}_2))$. The other two non-identity elements in our spanning set for $\mathcal{R}$ are easily seen to be either projections or scalar multiples of projections.

We close this note by relating our results on non-LP operator systems to the generalized Smith Ward problem regarding matrix ranges. We use the following fact, which is known to experts and appears in \cite[Proposition~11.2]{Ka14} in the three-dimensional case. As we were not able to find a proof in the literature, we provide one for the reader's convenience here.

\begin{proposition}
\label{proposition: generalized SWP and ucp lift}
Let $\cS=\text{span}\{I,S_1,...,S_m,S_1^*,...,S_m^*\}$ and $\cT=\text{span}\{I,T_1,...,T_m,T_1^*,...,T_m^*\}$ be finite-dimensional operator systems. Suppose that the linear map $\varphi:\cS \to \cT$ given by $\varphi(I)=I$, $\varphi(S_i)=T_i$ and $\varphi(S_i^*)=T_i^*$ for $i=1,...,m$ is well-defined. Then $\varphi$ is completely positive if and only if $W^n(T_1,...,T_m) \subseteq W^n(S_1,...,S_m)$ for every $n \in \bN$. In particular, $\varphi$ is a complete order isomorphism if and only if $W^n(S_1,...,S_m)=W^n(T_1,...,T_m)$ for all $n \in \bN$.
\end{proposition}

\begin{proof}
If $(A_1,...,A_m) \in W^n(T_1,...,T_m)$, then there is a ucp map $\Phi:\cT \to M_n$ such that $\Phi(T_i)=A_i$ for all $i=1,...,m$. Then if $\varphi:\cS \to \cT$ is ucp as above, the map $\Phi \circ \varphi:\cS \to M_n$ is ucp and $(\Phi \circ \varphi(S_1),...,\Phi \circ \varphi(S_m))=(A_1,...,A_m) \in W^n(S_1,...,S_m)$. This shows that, if $\varphi$ is completely positive, then $W^n(T_1,...,T_m) \subseteq W^n(S_1,...,S_m)$ for all $n$.

Conversely, suppose that $W^n(T_1,...,T_m) \subseteq W^n(S_1,...,S_m)$ for all $n$, and let $Z \in M_k(\cS)$ for some $k \in \bN$. Then by \cite[Lemma~4.1]{KPTT11}, $Z$ is positive in $M_k(\cS)$ if and only if $\psi^{(k)}(Z)$ is positive in $M_{nk}$ for every $n \in \bN$ and ucp $\psi:\cS \to M_n$. We write $Z=X_0 \otimes I+\sum_{i=1}^m (X_i \otimes S_i+Y_i \otimes S_i^*)$. Then $Z$ is positive in $M_k(\cS)$ if and only if $X_0 \otimes I+\sum_{i=1}^n (X_i \otimes \psi(S_i)+Y_i \otimes \psi(S_i)^*)$ is positive in $M_{nk}$ for every ucp map $\psi:\cS \to M_n$. This is the same as asking that $X_0 \otimes I+\sum_{i=1}^m (X_i \otimes A_i+Y_i \otimes A_i^*)$ is positive in $M_{nk}$ for every $(A_1,...,A_n) \in W^n(S_1,...,S_m)$. Thus, if $Z$ is positive in $M_k(\cS)$, then since $W^n(T_1,...,T_m) \subseteq W^n(S_1,...,S_m)$, for every $(B_1,...,B_m) \in W^n(T_1,...,T_m)$ we have that $X_0 \otimes I + \sum_{i=1}^n (X_i \otimes B_i+Y_i \otimes B_i^*)$ is positive in $M_{nk}$. Thus, $\varphi^{(k)}(Z)$ is positive in $M_k(\cT)$. Hence, $\varphi$ is completely positive. The last claim about complete order isomorphisms follows.
\end{proof}

For our purposes, we call a positive contraction $P \in \bofh$ an \textbf{essential projection} if its image $\pi(P) \in \qofh$ is an orthogonal projection. Re-phrasing Corollary \ref{corollary: subsystem of free group algebra} for $\mathbb{F}_2$ in terms of joint matrix ranges, we obtain the following.

\begin{corollary}
\label{corollary: self-adjoint partial isometry and two projections}
There exist essential projections $A,B,C,D$ in $\bofh$ with $\pi(A)\pi(B)=0$ such that, for any choice of compact operators $K_1,K_2,K_3 \in \mathcal{K}(\mathcal{H})$, there exists an $N \in \bN$ so that 
\[W^N(\pi(A)-\pi(B),\pi(C),\pi(D)) \neq W^N(A-B+K_1,C+K_2,D+K_3).\]
\end{corollary}

\begin{proof}
Let $\rho:C_r^*(\bF_2) \to \qofh$ be an injective unital $*$-homomorphism with no ucp lift to $\bofh$. Let $V=\begin{pmatrix} 0 & u_1 & 0 \\ u_1^* & 0 & u_2 \\ 0 & u_2^* & 0 \end{pmatrix}$, $R=\begin{pmatrix} 1 & 0 & 0 \\ 0 & 0 & 0 \\ 0 & 0 & 0 \end{pmatrix}$ and $S=\frac{1}{3}\begin{pmatrix} 1 & 1 & 1 \\ 1 & 1 & 1 \\ 1 & 1 & 1 \end{pmatrix}$. Then $V=P-Q$ for certain mutually orthogonal projections in $M_3(C_r^*(\bF_2))$, while $R$ and $S$ are also projections. Applying Corollary \ref{corollary: ext}, we obtain a self-adjoint partial isometry $T_1$ and two projections $T_2,T_3$ in $\qofh$ for which $\cS:=\text{span}\{I,T_1,T_2,T_3\}$ is completely order isomorphic to the operator system $\mathcal{R}$ from (\ref{subsystem of M3 of C_r*(F_2)}) and for which $\id:\cS \to \qofh$ has no ucp lift to $\bofh$. As any projection in $\qofh$ can be lifted to a positive contraction in $\bofh$, we may choose positive contractions $A,B,C,D \in \bofh$ such that $\pi(A)$ and $\pi(B)$ are projections in $\qofh$ with $\pi(A)\pi(B)=0$ and $T_1=\pi(A)-\pi(B)$, while $\pi(C)=T_2$ and $\pi(D)=T_3$. Suppose that there are $K_1,K_2,K_3 \in \cK(\cH)$ such that $W^n(T_1,T_2,T_3)=W^n(A-B+K_1,C+K_2,D+K_3)$ for all $n$. Define the unital, linear map $\varphi:\cS \to \bofh$ given by $\varphi(T_1)=A-B+K_1$, $\varphi(T_2)=C+K_2$ and $\varphi(T_3)=D+K_3$. Note that $\{I,T_1,T_2,T_3\}$ is linearly independent, so $\varphi$ is well-defined and ucp by Proposition \ref{proposition: generalized SWP and ucp lift}. Thus, $\varphi$ is a ucp lift of the identity map on $\cS$, which is a contradiction. The result follows.
\end{proof}

In other words, the generalized Smith-Ward problem fails even in the case of two projections and a self-adjoint isometry.

\section*{Acknowledgements}

The author was supported in part by a Lawrence Perko Faculty Research Award. The author would also like to thank David Blecher, Travis Russell, Vern Paulsen, Matthew Kennedy and Florin Pop for helpful conversations. The author also thanks an anonymous referee for many helpful suggestions.


\begin{thebibliography}{99}
\bib{An78}{article}{
author={J. Anderson},
title={A ${C}^{\ast}$-algebra for which ${E}xt(\mathcal{A})$ is not a group},
journal={Annals of Mathematics},
volume={107},
year={1978},
pages={455--458},
}
\bib{Ar69}{article}{
author={W. Arveson},
title={Subalgebras of ${C}^{\ast}$-algebras},
journal={Acta Mathematica},
volume={123},
year={1969},
pages={141--224},
}
\bib{Ar70}{article}{
author={W. Arveson},
title={Unitary invariants for compact operators},
journal={Bulletin of the American Mathematical Society},
volume={76},
year={1970},
pages={88--91},
}
\bib{Ar72}{article}{
author={W. Arveson},
title={Subalgebras of ${C}^{\ast}$-algebras II},
journal={Acta Mathematica},
volume={128},
year={1972},
pages={271--308},
}
\bib{Ar77}{article}{
author={W. Arveson},
title={Notes on extensions of ${C}^{\ast}$-algebras},
journal={Duke Journal of Mathematics},
volume={44},
number={2},
year={1977},
pages={329--355},
}
\bib{BO08}{book}{
author={N. Brown},
author={N. Ozawa},
title={${C}^{\ast}$-algebras and finite-dimensional approximations},
series={Graduate Studies in Mathematics}, publisher={American Mathematical Society Providence, RI}, 
year={2008},
}
\bib{Ch74}{article}{
author={M.-D. Choi},
title={A Schwarz inequality for positive linear maps on ${C}^{\ast}$-algebras},
journal={Illinois Journal of Mathematics},
volume={18},
year={1974},
pages={565--574},
}
\bib{CE77a}{article}{
author={M.-D. Choi},
author={E.G. Effros},
title={Injectivity and Operator Spaces},
journal={Journal of Functional Analysis},
volume={24},
year={1977},
pages={156--209},
}
\bib{CE77b}{article}{
author={M.-D. Choi},
author={E.G. Effros},
title={Lifting problems and the cohomology of ${C}^{\ast}$-algebras},
journal={Canadian Journal of Mathematics},
volume={XXIX},
number={5},
year={1977},
pages={1092--1111},
}
\bib{dlH00}{book}{
author={P. de la Harpe},
title={Topics in geometric group theory},
publisher={University of Chicago Press, Chicago, IL},
year={2000},
series={Chicago Lectures in Mathematics},
}
\bib{HT05}{article}{
author={U. Haagerup},
author={S. Thorbj\o rnsen},
title={A new application of random matrices: $\text{Ext}(C_{\text{red}}^*(F_2))$ is not a group},
journal={Annals of Mathematics},
volume={162},
year={2005},
pages={711--775},
}
\bib{KPTT11}{article}{
author={A.S. Kavruk},
author={V.I. Paulsen},
author={I.G. Todorov},
author={M. Tomforde},
title={Tensor products of operator systems},
journal={Journal of Functional Analysis},
volume={261},
number={2},
year={2011},
pages={267--299},
}
\bib{KPTT13}{article}{
author={A.S. Kavruk},
author={V.I. Paulsen},
author={I.G. Todorov},
author={M. Tomforde},
title={Quotients, exactness and nuclearity in the operator system category},
journal={Advances in Mathematics},
volume={235},
year={2013},
pages={321--360},
}
\bib{Ka14}{article}{
author={A.S. Kavruk},
title={Nuclearity related properties in operator systems},
journal={Journal of Operator Theory},
volume={71},
number={1},
year={2014},
pages={95--156},
}
\bib{Ka15}{article}{
author={A.S. Kavruk},
title={On a non-commutative analogue of a classical result of Namioka and Phelps},
journal={Journal of Functional Analysis},
volume={269},
year={2015},
pages={3282--3303},
}
\bib{Ki93}{article}{
author={E. Kirchberg},
title={On non-semisplit extensions, tensor products and exactness of group $C^{\ast}$-algebras},
journal={Inventiones mathematicae},
volume={112},
number={3},
year={1993},
pages={449--490},
}
\bib{LPP19}{article}{
author={C.-K. Li},
author={V.I. Paulsen},
author={Y.-T. Poon},
title={Preservation of the joint essential matricial range},
journal={Bulletin of the London Mathematical Society},
volume={51},
number={5},
year={2019},
pages={868--876},
}
\bib{LMH25}{article}{
author={L. Louder},
author={M. Magee},
author={W. Hide},
title={Strongly convergent unitary representations of limit groups},
journal={Journal of Functional Analysis},
volume={288},
number={6},
year={2025},
pages={110803},
}
\bib{Mu10}{article}{
author={V. M\"{u}ller},
title={The joint essential numerical range, compact perturbations, and the Olsen problem},
journal={Studia Mathematica},
volume={197},
year={2010},
pages={275--290},
}
\bib{NW82}{article}{
author={F. Narcowich},
author={J.D. Ward},
title={A characterization of essential matrix ranges},
journal={Bulletin of the London Mathematical Society},
volume={14},
year={1982},
pages={107--111},
}
\bib{OZ76}{article}{
author={C.L. Olsen},
author={W.R. Zame},
title={Some ${C}^{\ast}$-algebras with a single generator},
journal={Transactions of the American Mathematical Society},
volume={215},
year={1976},
pages={205--217},
}
\bib{PS79}{article}{
author={W.L. Paschke},
author={N. Salinas},
title={${C}^{\ast}$-algebras associated with free products of groups},
journal={Pacific Journal of Mathematics},
volume={82},
number={1},
year={1979},
pages={211--221},
}
\bib{Pa82}{article}{
author={V.I. Paulsen},
title={Preservation of essential matrix ranges by compact perturbations},
journal={Journal of Operator Theory},
volume={8},
number={2},
year={1982},
pages={299--317},
}
\bib{Pop20}{article}{
author={F. Pop},
title={Some $C^*$-algebras whose Ext is not a group},
journal={Operators and Matrices},
volume={14},
number={2},
year={2020},
pages={515--521},
}
\bib{SW80}{article}{
author={R.R. Smith},
author={J.D. Ward},
title={Matrix ranges for Hilbert space operators}, 
journal={American Journal of Mathematics},
volume={102},
year={1980}, 
pages={1031--1081},
}
\bib{Wa76}{article}{
author={S. Wassermann},
title={On tensor products of certain group ${C}^{\ast}$-algebras},
journal={Journal of Functional Analysis},
volume={23},
year={1976},
pages={239--254},
}
\end{thebibliography}
\end{document}